\documentclass[a4paper,reqno]{amsart}

\numberwithin{equation}{subsection}

\usepackage{amssymb}
\usepackage{amstext}
\usepackage{amsmath}
\usepackage{amscd}
\usepackage{amsthm}
\usepackage{amsfonts}
\usepackage{enumerate}
\usepackage{graphicx}
\usepackage{latexsym}
\usepackage{mathrsfs}
\usepackage[all,color]{xy}
\xyoption{all}
\usepackage[usenames,dvipsnames]{xcolor}

\usepackage{pstricks}
\usepackage{lscape}
\usepackage{comment}

\newtheorem{theorem}{Theorem}[section]

\newtheorem{corollary}[theorem]{Corollary}
\newtheorem{lemma}[theorem]{Lemma}

\newtheorem{definition-proposition}[theorem]{Definition-Proposition}

\theoremstyle{definition}
\newtheorem{definition}[theorem]{Definition}
\newtheorem{remark}[theorem]{Remark}

\newcommand{\ca}{{\mathcal A}}

\newcommand{\cc}{{\mathcal C}}

\newcommand{\cd}{\mathsf{D}}

\newcommand{\cf}{{\mathcal F}}
\newcommand{\ch}{{\mathcal H}}
\newcommand{\ck}{\mathsf{K}}
\newcommand{\cm}{{\mathcal M}}

\newcommand{\ct}{{\mathcal T}}
\newcommand{\cs}{{\mathcal S}}
\newcommand{\cp}{{\mathcal P}}
\newcommand{\cu}{{\mathcal U}}
\newcommand{\cx}{{\mathcal X}}
\newcommand{\cy}{{\mathcal Y}}
\newcommand{\cz}{{\mathcal Z}}

\renewcommand{\SS}{{\mathcal S}}

\newcommand{\Z}{\mathbb{Z}}

\newcommand{\bo}{\operatorname{b}\nolimits}

\renewcommand{\mod}{\mathsf{mod}}
\newcommand{\proj}{\mathsf{proj}}

\newcommand{\CM}{\mathsf{CM}}

\newcommand{\per}{\mathsf{per}}
\newcommand{\Ext}{\operatorname{Ext}\nolimits}
\newcommand{\Hom}{\operatorname{Hom}\nolimits}
\newcommand{\End}{\operatorname{End}\nolimits}

\newcommand{\op}{\operatorname{op}\nolimits}
\newcommand{\RHom}{\mathbf{R}\strut\kern-.2em\operatorname{Hom}\nolimits}

\newcommand{\Cok}{\operatorname{Cok}\nolimits}

\newcommand{\twoctilt}{2\strut\kern-.2em\operatorname{-ctilt}\nolimits}
\newcommand{\dctilt}{{\it d}\strut\kern-.2em\operatorname{-ctilt}\nolimits}

\newcommand{\thick}{\mathsf{thick}}

\newcommand{\add}{\mathsf{add}}

\begin{document}

\title[Quotients of triangulated categories]{Quotients of triangulated categories and\\
Equivalences of Buchweitz, Orlov and Amiot--Guo--Keller}

\author{Osamu Iyama}
\address{O. Iyama: Graduate School of Mathematics, Nagoya University, Chikusa-ku, Nagoya, 464-8602 Japan}
\email{iyama@math.nagoya-u.ac.jp}

\author{Dong Yang}
\address{D. Yang: Department of Mathematics, Nanjing University, 22 Hankou Road, Nanjing 210093, P. R. China}
\email{yangdong@nju.edu.cn}
\date{\today}
\begin{abstract}
We give a simple sufficient condition for a Verdier quotient $\ct/\cs$ of a triangulated category $\ct$ by a thick subcategory $\cs$ to be realized inside of $\ct$ as an ideal quotient.
As applications, we deduce three significant results by Buchweitz, Orlov and Amiot--Guo--Keller.
\end{abstract}
\maketitle

\section{Introduction}
Triangulated categories are ubiquitous in mathematics, appearing in various areas such as representation theory, algebraic geometry, algebraic topology and mathematical physics. A fundamental tool to construct new triangulated categories from given ones is to take Verdier quotients, for example, derived categories are certain Verdier quotients of homotopy categories. However, Verdier quotient categories are in general  hard to understand because taking Verdier quotients could drastically change morphisms. Generally it is even hard to know whether morphisms between two objects in a Verdier quotient form sets.

The first aim of this paper is to give a simple sufficient condition for a Verdier quotient $\ct/\cs$ of a triangulated category $\ct$ by a thick subcategory $\cs$ to be realized inside of $\ct$ as an ideal quotient $\cz/[\cp]$ for certain explicitly constructed full subcategories $\cz\supset\cp$ of $\ct$ (Theorem \ref{equivalence}).
Such a realization is very helpful in studying $\ct/\cs$ since the morphism sets in the ideal quotient $\cz/[\cp]$ are very easy to control. For example, in this case if $\ct$ is Hom-finite over a field and Krull--Schmidt, so is $\ct/\cs$.
The second aim of this paper is to show that the following three significant results can be regarded as special cases of our realization.
\begin{enumerate}[$\bullet$]
\item The first one is Buchweitz's equivalence \cite{Buchweitz87,Rickard,KellerVossieck} between the singularity category of an Iwanaga--Gorenstein ring  and the stable category of Cohen--Macaulay modules over the ring. In fact, we recover the silting reduction introduced in \cite{IyamaYang} more generally (Corollaries \ref{silting reduction}, \ref{buchweitz}).
\item The second one is Orlov's theorem \cite{Orlov} relating the graded singularity category of a $\mathbb{Z}$-graded Iwanaga--Gorenstein ring and the derived category of the corresponding noncommutative projective scheme (Corollary \ref{orlov3}). It gives a direct connection between projective geometry and Cohen--Macaulay representations.
\item The third one is Amiot--Guo--Keller's equivalence \cite{Amiot09,Guolingyan11a,IyamaYang} which plays an important role in the categorification of Fomin--Zelevinsky's cluster algebras \cite{Keller}. It realizes the cluster category as the fundamental domain in the perfect derived category of Ginzburg dg algebras 
(Corollary \ref{amiot-guo-keller-CY}).
\end{enumerate}

In fact, the third application is given in a wider setting. We introduce the notion of a \emph{relative Serre quadruple} as a certain pair of a nice triangulated category $\ct$ and its thick subcategory $\cs$ with extra data (Definition~\ref{serre quadruple}). Then the corresponding \emph{AGK category} is defined as the Verdier quotient $\ct/\cs$ (Definition~\ref{define amiot-guo-keller}).
This is a wide generalization of cluster categories, and we prove that AGK category is equivalent to the \emph{fundamental domain}, a certain full subcategory of $\ct$ given explicitly (Theorem~\ref{amiot-guo-keller}).

\subsection{Preliminaries}

Let $\ct$ be a triangulated category, and $\cx$ and $\cy$ full subcategories of $\ct$. We denote by $\cx*\cy$ the full subcategory of $\ct$ consisting of objects $T\in\ct$ such that there exists a triangle $X\to T\to Y\to X[1]$ with $X\in\cx$ and $Y\in\cy$. When $\Hom_{\ct}(\cx,\cy)=0$ holds, we write $\cx*\cy=\cx\perp\cy$. For full subcategories $\cx_1,\ldots,\cx_n$, we define $\cx_1*\cdots*\cx_n$ and $\cx_1\perp\cdots\perp\cx_n$ inductively. 
We write
\[\cx^\perp:=\{T\in\ct\mid\Hom_{\ct}(\cx,T)=0\}\ \mbox{ and }\ {}^\perp\cx:=\{T\in\ct\mid\Hom_{\ct}(T,\cx)=0\}.\]
When $\ct=\cx\perp\cy$, $\cx={}^\perp\cy$ and $\cy=\cx^\perp$ hold, we say that $\ct=\cx\perp\cy$ is a \emph{torsion pair} of $\ct$ \cite{IyamaYoshino08}.
If a torsion pair $\ct=\cx\perp\cy$ satisfies $\cx[1]\subset\cx$ (respectively, $\cx[1]\supset\cx$, $\cx[1]=\cx$), then we call it a \emph{t-structure} \cite{BeilinsonBernsteinDeligne} (respectively, \emph{co-t-structure} \cite{Pauksztello08,Bondarko10}, \emph{stable t-structure} \cite{Miyachi}). 
In the literature a t-structure (respectively, co-t-structure) usually means the pair $(\cx,\cy[1])$ (respectively, $(\cx,\cy[-1]))$. Our convention is more suitable for our purpose.
If $\ct=\cx_1\perp\cdots\perp\cx_n$ for thick subcategories $\cx_1,\ldots,\cx_n$ of $\ct$, we say that $\ct=\cx_1\perp\cdots\perp\cx_n$
is a \emph{weak semi-orthogonal decomposition} of $\ct$ \cite{Orlov}.
Note that it is often written as $\langle\cx_n,\ldots,\cx_1\rangle$ in the literature \cite{Huybrechts}.

\subsection{Main results}

Our main result is given under the following very simple axioms.

\begin{itemize}
\item[(T0)] $\ct$ is a triangulated category, $\cs$ is a thick subcategory of $\ct$ and $\cu=\ct/\cs$.
\item[(T1)] $\cs$ has a torsion pair $\cs=\cx\perp\cy$.
\item[(T2)] $\ct$ has torsion pairs $\ct=\cx\perp\cx^\perp={}^\perp\cy\perp\cy$.
\end{itemize}
Notice that $\cx$ and $\cy$ are not necessarily triangulated subcategories of $\ct$, and this is important in our applications. In this setting, we define full subcategories of $\ct$ by
\[\cz:=\cx^\perp\cap{}^\perp\cy[1]\ \mbox{ and }\ \cp:=\cx[1]\cap\cy.\]
We denote by $\cz/[\cp]$ the additive category with the same objects as $\cz$ and
\[\Hom_{\cz/[\cp]}(X,Y)=\Hom_{\ct}(X,Y)/[\cp](X,Y)\]
for $X,Y\in\cz$, where $[\cp](X,Y)$ is the subgroup of $\Hom_{\ct}(X,Y)$ which consists of morphisms factoring through objects in $\cp$.

The first main result in this paper enables us to realize the Verdier quotient $\cu=\ct/\cs$ as the ideal quotient $\cz/[\cp]$.

\begin{theorem}\label{equivalence}
Under the assumptions {\rm(T0)}, {\rm(T1)} and {\rm(T2)}, the composition $\cz\subset\ct\to\cu$ of natural functors
induces an equivalence of additive categories $\cz/[\cp]\simeq\cu$.
In particular, the category $\cz/[\cp]$ has a structure of a triangulated category.
\end{theorem}

Note that if $\cs=\cx\perp\cy$ is a t-structure, then $\cp=0$. In Section~\ref{apply Amiot-Guo-Keller}, we apply Theorem \ref{equivalence} to AGK categories. In particular, we deduce Amiot--Guo--Keller's equivalence (Corollary \ref{amiot-guo-keller-CY}) from Theorem~\ref{equivalence}.

\medskip
Next we consider the following special case of  (T1).
\begin{itemize}
\item[(T1$^\prime$)] $\cs=\cx\perp\cy$ is a co-t-structure.
\end{itemize}
In this case, we have the following direct description of the triangulated structure of $\cz/[\cp]$ , which is an analog of the triangulated structures introduced in \cite{Happel,IyamaYoshino08} in different settings.

\begin{theorem}\label{triangle equivalence}
Under the assumptions {\rm(T0)}, {\rm(T1$^\prime$)} and {\rm(T2)}, we have the following.
\begin{itemize}
\item[(a)] The shift functor and triangles of the triangulated category $\cz/[\cp]$ are the following.
\begin{itemize}
\item[$\bullet$] For $X\in\cz$, we take a triangle
\[X\xrightarrow{\iota_X}P_X\to X\langle1\rangle\to X[1]\]
with a (fixed) left $\cp$-approximation $\iota_X$.
Then $\langle1\rangle$ gives a well-defined auto-equivalence of $\cz/[\cp]$, which is the shift functor of $\cz/[\cp]$.
\item[$\bullet$] For a triangle $X\xrightarrow{f} Y\xrightarrow{g} Z\xrightarrow{h}X[1]$ in $\ct$
with $X,Y,Z\in\cz$, take the following commutative diagram of triangles:
\[
\xymatrix{
X\ar[r]^f\ar@{=}[d]&Y\ar[r]^g\ar[d]&Z\ar[r]^h\ar[d]^a&X[1]\ar@{=}[d]\\
X\ar[r]^{\iota_X}&P_X\ar[r]&X\langle1\rangle\ar[r]&X[1]
}\]
The triangles in $\cz/[\cp]$ are the diagrams which is isomorphic to the image of $X\xrightarrow{f}Y\xrightarrow{g}Z\xrightarrow{a}X\langle1\rangle$ in $\cz/[\cp]$.
\end{itemize}
\item[(b)] We have $\ct=\cx\perp\cz\perp\cy[1]$.
\end{itemize}
\end{theorem}

In Section~\ref{apply Buchweitz}, we deduce Buchweitz's equivalence (Corollary~\ref{buchweitz}) and the silting reduction (Corollary~\ref{frobenius case}) from Theorem~\ref{triangle equivalence}.

\medskip
Finally we consider the following further special case of (T1$^\prime$).
\begin{itemize}
\item[(T1$^{\prime\prime}$)] $\cs=\cx\perp\cy$ is a stable t-structure.
\end{itemize}
In this case, $\cx$, $\cy$ and $\cz$ are thick subcategories of $\ct$, and $\cp=0$ holds.
As a consequence, we deduce the following result immediately.

\begin{corollary}\label{triangle equivalence2}
Under the assumptions {\rm(T0)}, {\rm(T1$^{\prime\prime}$)} and {\rm(T2)}, we have a weak semi-orthogonal decomposition $\ct=\cx\perp\cz\perp\cy$. In particular, the composition $\cz\subset\ct\to\cu$ is a triangle equivalence.
\end{corollary}

In Section~\ref{apply Orlov}, we deduce Orlov's theorem (Corollary \ref{orlov3}) from Corollary~\ref{triangle equivalence2}.

\medskip
Let us explain the structure of this paper. Our main
Theorems \ref{equivalence} and \ref{triangle equivalence} will be proved in the last Section~\ref{s:proofs}. In the next Section~\ref{s:applications}, we deduce the three applications: Buchweitz's equivalence, Orlov's equivalences, and Amiot--Guo--Keller's equivalence.

During the preparation of this paper, the authors were informed that
there are analogous results to our Theorem \ref{equivalence} in different settings.
One is `Hovey's twin cotorsion pairs' due to Nakaoka \cite{Nakaoka},
and the other is `additive categories with additive endofunctors' due to Li \cite{Li16}.
Their arguments are more involved than ours since they give direct descriptions of the triangulated structure of $\cz/[\cp]$ in the setting of Theorem 1.1.
It will be interesting to have a closer look at the connections between these results.

\medskip\noindent{\bf Ackowledgements}
Results in this paper were presented at the conference ``Cluster Algebras and Geometry'' in M\"unster in March 2016. The first author thanks Karin Baur and Lutz Hille for their hospitality. He also thanks Hiroyuki Nakaoka for explaining his results. The authors thank Gustavo Jasso for pointing out an error in the proof of Corollary~\ref{frobenius case}.
He is supported by JSPS Grant-in-Aid for Scientific Research (B) 16H03923, (C) 23540045 and (S) 15H05738.
The second author is supported by the National Science Foundation of China No. 11401297.

\section{Applications of main results}\label{s:applications}

In this section we give three applications of Theorems~\ref{equivalence} and \ref{triangle equivalence} and Corollary~\ref{triangle equivalence2}.

\subsection{Silting reduction and Buchweitz's equivalence}\label{apply Buchweitz}

Tilting theory is powerful to control equivalences of derived categories, and silting objects/subcategories are central in tilting theory. It is shown in
\cite{IyamaYang,Wei15} that the Verdier quotient of a triangulated category by its thick subcategory with a silting subcategory can be realized as an ideal quotient (\emph{silting reduction}).
The aim of this subsection is to deduce silting reduction from our main Theorems~\ref{equivalence} and \ref{triangle equivalence}. 

Recall that a full subcategory $\cp$ of a triangulated category $\ct$ is \emph{presilting} if $\Hom_{\ct}(\cp,\cp[>\hspace{-3pt}0])=0$. A presilting subcategory $\cp$ is called \emph{silting} if the thick subcategory $\thick\cp$ generated by $\cp$ is $\ct$.

Now we assume the following.
\begin{itemize}
\item[(P0)] $\ct$ is a triangulated category, $\cp$ is a presilting subcategory of $\ct$ such that $\cp=\add\cp$, $\cs=\thick\cp$, and $\cu=\ct/\cs$.
\item[(P1)] $\cp$ is covariantly finite in ${}^\perp\cp[>\hspace{-3pt}0]$ and contravariantly finite in $\cp[<\hspace{-3pt}0]^{\perp}$.
\item[(P2)] For any $X\in\ct$, we have $\Hom_{\ct}(X,\cp[\ell])=0=\Hom_\ct(\cp,X[\ell])$ for $\ell\gg0$.
\end{itemize}
As a special case of our Theorems~\ref{equivalence} and \ref{triangle equivalence}, we recover the following silting reduction. 

\begin{corollary}[{\cite[Theorems 3.1 and 3.6]{IyamaYang}}]\label{silting reduction}
Under the assumptions {\rm(P0)}, {\rm(P1)} and {\rm(P2)},  let
\begin{eqnarray*}
&{\displaystyle\cx:=\bigcup_{i>0}\cp[-i]*\cdots*\cp[-2]*\cp[-1],}&\cy:=\bigcup_{i>0}\cp*\cp[1]*\cdots*\cp[i],\\
&\cz:=\cx^\perp\cap{}^\perp\cy[1]=\cp[<\hspace{-3pt}0]^\perp\cap{}^\perp\cp[>\hspace{-3pt}0].&
\end{eqnarray*}
Then the following assertion holds.
\begin{itemize}
\item[(a)] {\rm(T0)}, {\rm(T1$^\prime$)} and {\rm(T2)} in Theorem \ref{triangle equivalence} are satisfied.
\item[(b)] We have a triangle equivalence $\cz/[\cp]\simeq\cu$, where the structure of a triangulated category of $\cz/[\cp]$ is described in Theorem \ref{triangle equivalence}.
\item[(c)] We have $\ct=\cx\perp\cz\perp\cy[1]$.
\end{itemize}
\end{corollary}

\begin{proof}
It is well-known that we have a co-t-structure $\cs=\cx\perp\cy$ (see for example \cite[Proposition 2.8]{IyamaYang}).
By \cite[Proposition 3.2]{IyamaYang}, we have torsion pairs $\ct=\cx\perp\cx^\perp={}^\perp\cy\perp\cy$.
Thus (T0), (T1$^\prime$) and (T2) in Theorem \ref{triangle equivalence} are satisfied, and we have the assertions.
\end{proof}

Now we apply Corollary~\ref{silting reduction} to prove Keller--Vossieck's equivalence \cite{KellerVossieck} in our context.
For an additive category $\ca$, we denote by $\ck(\ca)$ the homotopy category of complexes on $\ca$, and by $\ck^{\bo}(\ca)$ the full subcategory consisting of bounded complexes.

Let $\cf$ be a Frobenius category, $\cp$ the category of projective-injective objects in $\cf$ and $\underline{\cf}=\cf/[\cp]$ the stable category of $\cf$.
Then $\underline{\cf}$ has a structure of a triangulated category due to Happel \cite{Happel}.
We denote by $\ck^{-,\bo}(\cp)$ the full subcategory of $\ck(\cp)$ consisting of complexes $X=(X^i,d^i\colon X^i\to X^{i+1})$ satisfying the following conditions.
\begin{itemize}
\item[(a)] There exists $n_X\in\Z$ such that $X^i=0$ holds for each $i>n_X$.
\item[(b)] There exist $m_X\in\Z$ and a conflation $0\to Y^{i-1}\xrightarrow{a^{i-1}}X^{i}\xrightarrow{b^{i}}Y^{i}\to0$ in $\cf$ for each $i\le m_X$ such that $d^i=a^ib^i$ holds for each $i<m_X$.
\end{itemize}
It is elementary that the category $\cf$ is equivalent to the full subcategory of $\ck^{-,\bo}(\cp)$ consisting of complexes which are isomorphic in $\ck(\cp)$ to some $X$ satisfying $n_X\le0\le m_X$, and we identify them.
We denote by $\ck^{>0}(\cp)$ (respectively, $\ck^{<0}(\cp)$) the full subcategory of $\ck^{\bo}(\cp)$ consisting of complexes $X=(X^i,d^i\colon X^i\to X^{i+1})$ satisfying $X^i=0$ for $i\le0$ (respectively, $i\ge0$).

\begin{corollary}\label{frobenius case}
Let $\cf$ be a weakly idempotent complete Frobenius category with the category $\cp$ of projective-injective objects.
Then we have a decomposition
\[\ck^{-,\bo}(\cp)=\ck^{>0}(\cp)\perp\cf\perp\ck^{<0}(\cp).\]
Moreover the composition $\cf\subset\ck^{-,\bo}(\cp)\to\ck^{-,\bo}(\cp)/\ck^{\bo}(\cp)$ induces a triangle equivalence
\[\underline{\cf}\xrightarrow{\simeq}\ck^{-,\bo}(\cp)/\ck^{\bo}(\cp).\]
\end{corollary}

The `Moreover' part is contained in \cite[Example 2.3]{KellerVossieck}.

\begin{proof}
Let $\ct=\ck^{-,\bo}(\cp)$, $\cs=\ck^{\bo}(\cp)$ and $\cu=\ct/\cs$. Then the condition (P0) is clearly satisfied.
We show that the conditions (P1) and (P2) are satisfied.

(P2) Fix $X\in\ct$. Then the condition (a) implies $\Hom_{\ct}(\cp[<\hspace{-3pt}-n_X],X)=0$, and the condition (b) implies $\Hom_{\ct}(X,\cp[>\hspace{-3pt}-m_X])=0$. Thus the condition (P2) is satisfied.

(P1) Fix $X\in\cp[<\hspace{-3pt}0]^\perp$ and put $n:=n_X$. 
If $n\le0$, then the natural morphism $X^0\to X$ gives a right $\cp$-approximation of $X$.
If $n>0$, then the natural morphism $X^n[-n]\to X$ must be zero in $\ct$. Thus there exists $f\in\Hom_{\cp}(X^n,X^{n-1})$ such that $1_{X^n}=d^{n-1}f$.
\[\xymatrix{
X^n[-n]:&&&X^n\ar[d]^{1_{X^n}}\ar[dl]_{f}\\
X:&\cdots\ar[r]&X^{n-1}\ar[r]_{d^{n-1}}&X^n\ar[r]&0\ar[r]&\cdots
}\]
Since $\cf$ is weakly idempotent complete, $X$ is isomorphic to the complex $X'$ obtained by replacing $d^{n-1}:X^{n-1}\to X^n$ in $X$ by $\Cok f\to 0$. Then $n_{X'}<n=n_X$ holds.
Repeating the same argument, we can assume that $n_X\le0$ without loss of generality. 
Thus $\cp$ is contravariantly finite in $\cp[<\hspace{-3pt}0]^\perp$.

Fix $X\in{}^\perp\cp[>\hspace{-3pt}0]$ and $m:=m_X$. Consider the conflation $0\to Y^{i-1}\xrightarrow{a^{i-1}}X^{i}\xrightarrow{b^{i}}Y^{i}\to0$ given in the condition (b). 
If $m\ge0$, then $b^0:X^0\to Y^0$ is a cokernel of $d^{-1}$. Thus the composition of the natural morphism $X\to Y^0$ and an inflation $Y^0\to P$ with $P\in\cp$ gives a left $\cp$-approximation of $X$.

Assume $m<0$, and take an inflation $f:Y^{m}\to P$ in $\cf$ with $P\in\cp$.
Since $b^m:X^m\to Y^m$ is a cokernel of $d^{m-1}$, there exists $a^m:Y^m\to X^{m+1}$ such that $d^m=a^mb^m$.
Since $X\in{}^\perp\cp[>\hspace{-3pt}0]$, the composition of natural morphisms $X\to Y^m[-m]\to P[-m]$ must be zero in $\ct$.
Thus there exists $g\in\Hom_{\cp}(X^{m+1},P)$ such that $fb^m=gd^m$.
\[\xymatrix@R1em{
X:&\cdots\ar[r]&X^{m-1}\ar[r]^{d^{m-1}}&X^{m}\ar[rr]^{d^{m}}\ar[dr]|{b^m}\ar[dd]_{fb^m}&&X^{m+1}\ar[r]^{d^{m+1}}\ar@/^.5cm/[ddll]^g&\cdots\\
&&&&Y^m\ar[ur]|{a^m}\ar[dl]|f\\
P[-m]:&&&P
}\]
Then $f=ga^m$ holds. Since $f$ is an inflation in $\cf$ and $\cf$ is weakly idempotent complete, $a^m$ is also an inflation in $\cf$. Thus there exists a conflation $0\to Y^{m}\xrightarrow{a^{m}}X^{m+1}\xrightarrow{b^{m+1}}Y^{m+1}\to0$ in $\cf$ satisfying $d^m=a^mb^m$, and we can replace $m_X$ by $m_X+1$.
Repeating the same argument, we can assume $m_X\ge0$ without loss of generality.
Thus $\cp$ is covariantly finite in ${}^\perp\cp[>\hspace{-3pt}0]$.
Thus the condition (P1) is satisfied.

We are ready to complete the proof of Corollary~\ref{frobenius case}.
Thanks to Corollary~\ref{silting reduction}, it suffices to show that $\cz=\cp[<\hspace{-3pt}0]^\perp\cap{}^\perp\cp[>\hspace{-3pt}0]$ coincides with $\cf$, and that the triangulated structure of $\cz/[\cp]$ coincides with that of $\underline{\cf}=\cf/[\cp]$.
The inclusion $\cz\supset\cf$ is clear.
Conversely, the above argument shows that any object in $\cz$ is isomorphic to some $X$ with $n_X\le0\le m_X$, and hence belongs to $\cf$. Thus $\cz=\cf$ holds.
On the other hand, the triangulated structure of $\cz/[\cp]$ described in Theorem~\ref{triangle equivalence}(a) is nothing but Happel's triangulated structure of $\underline{\cf}$ in this setting.
Thus the claim follows.
\end{proof}

Now we apply Corollary~\ref{frobenius case} to prove Buchweitz's equivalence \cite{Buchweitz87,Rickard}.
Recall that a Noetherian ring $R$ is called an \emph{Iwanaga--Gorenstein ring} if ${\rm inj.dim}_RR<\infty$ and ${\rm inj.dim}R_R<\infty$. Let
\[\CM R:=\{X\in\mod R\mid\Ext^i_R(X,R)=0\ \forall i>0\}\]
be the category of \emph{Cohen--Macaulay $R$-modules}.
For an additive category $\ca$, we denote by $\ck^{\bo}(\ca)=\ck^{>0}(\ca)\perp\ck^{\le0}(\ca)$ the standard co-t-structure (see \cite[Section 1.1]{Bondarko10} and \cite[Section 3.1]{Pauksztello08}).
For an abelian category $\ca$, we denote by $\cd^{\bo}(\ca)$ the bounded derived category of $\ca$.

\begin{corollary}\label{buchweitz}
Let $R$ be an Iwanaga--Gorenstein ring. Then we have
\[\cd^{\bo}(\mod R)=\ck^{>0}(\proj R)\perp\CM R\perp\ck^{<0}(\proj R).\]
Moreover the composition $\CM R\subset\cd^{\bo}(\mod R)\to\cd^{\rm b}(\mod R)/\ck^{\rm b}(\proj R)$ induces a triangle equivalence
\[\underline{\CM} R\simeq\cd^{\rm b}(\mod R)/\ck^{\rm b}(\proj R).\]
\end{corollary}

The `Moreover' part is \cite[Theorem 4.4.1(2)]{Buchweitz87}, and \cite[Theorem 2.1]{Rickard} for self-injective algebras.

\begin{proof}
Any complex of finitely generated projective $R$-modules with bounded cohomologies is in $\ck^{-,\bo}(\proj R)$. It follows that $\cd^{\bo}(\mod R)$ is triangle equivalent to $\ck^{-,\bo}(\proj R)$. The desired results are obtained by applying Corollary~\ref{frobenius case} to $\cf=\CM R$.
\end{proof}

\subsection{Orlov's equivalences}\label{apply Orlov}

Throughout this subsection, we assume the following.
\begin{itemize}
\item[(R0)] $R$ is a $\Z$-graded Iwanaga--Gorenstein ring such that $R=R_{\ge0}$.
\end{itemize}

In \cite{Orlov}, Orlov gave a remarkable connection between two Verdier quotients of $\cd^{\bo}(\mod^{\Z}R)$, where $\mod^{\Z}R$ is the category of $\Z$-graded finitely generated $R$-modules. One is $\cd^{\bo}(\mod^{\Z}R)/\ck^{\bo}(\proj^{\Z}R)$ for the category $\proj^{\Z}R$ of $\Z$-graded finitely generated projective $R$-modules.
This is important in Cohen--Macaulay representation theory as we saw in the previous subsection.
The other is $\cd^{\bo}(\mod^{\Z}R)/\cd^{\bo}(\mod_0^{\Z}R)$ for the category $\mod_0^{\Z}R$ of $\Z$-graded $R$-modules of finite length. This is important in commutative and non-commutative algebraic geometry.

The aim of this subsection is to deduce Orlov's theorem from Corollary~\ref{triangle equivalence2} in a slightly more general setting than Orlov's original setting \cite{Orlov}.

For an integer $\ell$, we denote by $\mod^{\ge\ell}R$ (respectively, $\mod^{\le\ell}R$) the full subcategory of $\mod^{\Z}R$ consisting of all $X$ satisfying $X_i=0$ for any $i<\ell$ (respectively, $i>\ell$). 
We denote by $\proj^{\ge\ell}R$ (respectively, $\proj^{\le\ell}R$) the full subcategory of $\proj^{\Z}R$ consisting of all $P$ which are generated by homogeneous elements of degrees at least $\ell$ (respectively, at most $\ell$). Let $\mod^{>\ell}R:=\mod^{\ge\ell+1}R$, $\mod^{<\ell}R:=\mod^{\le\ell-1}R$, $\proj^{>\ell}R:=\proj^{\ge\ell+1}R$ and $\proj^{<\ell}R:=\proj^{\le\ell-1}R$.
Since $R$ is Iwanaga--Gorenstein, we have a duality \cite[Corollary 2.11]{Miyachi2}
\[(-)^*:=\RHom_R(-,R):\cd^{\bo}(\mod^{\Z}R)\leftrightarrow\cd^{\bo}(\mod^{\Z}R^{\op}).\]
Since the extension groups in $\mod^{\ge\ell}R$ (respectively, $\mod^{\le\ell}R$) coincide with those in $\mod^{\Z}R$, we can regard $\cd^{\bo}(\mod^{\ge\ell}R)$ (respectively, $\cd^{\bo}(\mod^{\le\ell}R)$) as a thick subcategory of $\cd^{\bo}(\mod^{\Z}R)$.

Under the following condition, we apply Corollary~\ref{triangle equivalence2} to deduce the following result.
\begin{itemize}
\item[(R1)] $R_0$ has finite global dimension.
\end{itemize}

\begin{corollary}\label{orlov1}
Let $\ell\in\Z$. Under the assumptions {\rm(R0)} and {\rm(R1)}, we have
\[\cd^{\bo}(\mod^{\Z}R)=\ck^{\bo}(\proj^{<\ell}R)\perp\left(\cd^{\bo}(\mod^{\ge\ell}R)\cap\cd^{\bo}(\mod^{>-\ell}R^{\op})^*\right)\perp\ck^{\bo}(\proj^{\ge\ell}R).\]
In particular, the composition
\[\cd^{\bo}(\mod^{\ge\ell}R)\cap\cd^{\bo}(\mod^{>-\ell}R^{\op})^*\subset\cd^{\bo}(\mod^{\Z}R)\to\cd^{\bo}(\mod^{\Z}R)/\ck^{\bo}(\proj^{\Z}R)\]
is a triangle equivalence.
\end{corollary}

This is given implicitly in \cite[Lemma 2.4]{Orlov}, and a similar result is given in \cite[Theorem 3.17]{HIMO}.

To prove it, we need the following elementary observation.

\begin{lemma}[{c.f.\ \cite[Lemma 2.3]{Orlov}}]\label{SOD of K}
Let $\ell\in\Z$. Under the assumption {\rm(R0)}, we have
$\ck^{*}(\proj^{\Z}R)=\ck^{*}(\proj^{<\ell}R)\perp\ck^{*}(\proj^{\ge\ell}R)$ for $*=$nothing, $\pm$, $\bo$.
If {\rm(R1)} is also satisfied, then $\ck^{-,\bo}(\proj^{\Z}R)=\ck^{\bo}(\proj^{<\ell}R)\perp\ck^{-,\bo}(\proj^{\ge\ell}R)$.
\end{lemma}

\begin{proof}
Clearly $\Hom_{\ck(\proj^{\Z}R)}(\ck(\proj^{<\ell}R),\ck(\proj^{\ge\ell}R))=0$ holds.
For any $P\in\proj^{\Z}R$, we denote by $P^{<\ell}$ the sub-$R$-module of $P$
generated by the subspace $\bigoplus_{i<\ell}P_i$, and $P^{\ge\ell}:=P/P^{\le\ell}$.
These give functors $(-)^{<\ell}:\proj^{\Z}R\to\proj^{<\ell}R$,
$(-)^{\ge\ell}:\proj^{\Z}R\to\proj^{\ge\ell}R$ and a sequence
\[0\to(-)^{<\ell}\to {\rm id}\to(-)^{\ge\ell}\to0\]
of natural transformations which is objectwise split exact.
Therefore we have induced triangle functors
$(-)^{<\ell}:\ck(\proj^{\Z}R)\to\ck(\proj^{<\ell}R)$ and
$(-)^{\ge\ell}:\ck(\proj^{\Z}R)\to\ck(\proj^{\ge\ell}R)$, which preserve the boundedness condition $*$, and a functorial triangle
$Q^{<\ell}\to Q\to Q^{\ge\ell}\to Q^{<\ell}[1]$ for any $Q\in\ck(\proj^{\Z}R)$.
Thus we have the first equality.

We show the second equality. 
Clearly $\Hom_{\ck(\proj^{\Z}R)}(\ck^{\bo}(\proj^{<\ell}R),\ck^{-,\bo}(\proj^{\ge\ell}R))=0$ holds.
We need to show that any $Q\in\ck^{-,\bo}(\proj^{\Z}R)$ belongs to $\ck^{\bo}(\proj^{<\ell}R)\perp\ck^{-,\bo}(\proj^{\ge\ell}R)$.
It suffices to show that $Q^{<\ell}$ belongs to $\ck^{\bo}(\proj^{<\ell}R)$.
Let $Q^{(i)}\in\ck^{-,\bo}(\proj^{\Z}R)$ be a projective resolution of $H^i(Q)\in\mod^\Z R$. 
Take $m\gg0$ such that $H^i(Q)=0$ for all $i\in\Z$ with $|i|>m$.
Then $Q$ belongs to $Q^{(-m)}[m]*Q^{(1-m)}[m-1]*\cdots*Q^{(m)}[-m]$.
By the condition (R1), $(Q^{(i)})^{<\ell}$ belongs to $\ck^{\bo}(\proj^{<\ell}R)$. 
Thus $Q^{<\ell}$ belongs to $(Q^{(-m)})^{<\ell}[m]*\cdots*(Q^{(m)})^{<\ell}[-m]\subset\ck^{\bo}(\proj^{<\ell}R)$, as desired.
\end{proof}

\begin{proof}[Proof of Corollary \ref{orlov1}]
Without loss of generality, assume $\ell=0$.
Let $\ct:=\cd^{\bo}(\mod^{\Z}R)$, $\cs:=\ck^{\bo}(\proj^{\Z}R)$, $\cx:=\ck^{\bo}(\proj^{<0}R)$ and $\cy:=\ck^{\bo}(\proj^{\ge0}R)$. By Lemma \ref{SOD of K}, we have stable t-structures
\begin{eqnarray*}
\ck^{\bo}(\proj^{\Z}R)=\cx\perp\cy\ \mbox{ and }\ \cd^{\bo}(\mod^{\Z}R)\simeq\ck^{-,\bo}(\proj^{\Z}R)=\cx\perp\cx^\perp
\end{eqnarray*}
with $\cx^\perp=\ck^{-,\bo}(\proj^{\ge0}R)\simeq\cd^{\bo}(\mod^{\ge0}R)$.
Replacing $R$ by $R^{\op}$ and shifting the degree in the second stable t-structure, we have a stable t-structure
\[\cd^{\bo}(\mod^{\Z}R^{\op})=\ck^{\bo}(\proj^{\le0}R^{\op})\perp\cd^{\bo}(\mod^{>0}R^{\op}).\]
Applying $(-)^*$ and using $\ck^{\bo}(\proj^{\le0}R^{\op})^*=\cy$, we have a stable t-structure
\[\cd^{\bo}(\mod^{\Z}R)=\cd^{\bo}(\mod^{>0}R^{\op})^*\perp\ck^{\bo}(\proj^{\le0}R^{\op})^*={}^\perp\cy\perp\cy.\]
Thus (T0), (T1$^{\prime\prime}$) and (T2) in Corollary~\ref{triangle equivalence2} are satisfied, and we have the assertion.
\end{proof}

For an integer $\ell$, let $\mod_0^{\ge\ell}R:=\mod_0^{\Z}R\cap\mod^{\ge\ell}R$ and $\mod_0^{\le\ell}R:=\mod_0^{\Z}R\cap\mod^{\le\ell}R$. It is clear that $\mod_0^{\le\ell}R=\mod^{\le\ell}R$ holds.
Let $\mod^\ell R:=(\mod^{\ge\ell}R)\cap(\mod^{\le\ell}R)$.
As before, we can regard $\cd^{\bo}(\mod_0^{\Z}R)$ (respectively, $\cd^{\bo}(\mod_0^{\ge\ell}R)$, $\cd^{\bo}(\mod_0^{\le\ell}R)$, $\cd^{\bo}(\mod_0^{\ell}R)$) as a thick subcategory of $\cd^{\bo}(\mod^{\Z}R)$.

The following conditions are crucial for our next result.

\begin{itemize}
\item[(R2)] $R_0$ is an Artinian ring. 
\item[(R3)] There exists $a\in\Z$ such that $(-)^*$ restricts to a duality $(-)^*:\cd^{\bo}(\mod^0R)\leftrightarrow\cd^{\bo}(\mod^aR^{\op})$.
\end{itemize}

For a Noetherian ring $R$, the condition (R2) implies that $R_i$ has finite length as an $R$-module and as an $R^{\op}$-module for any $i\in\Z$ (here the $R$- and $R^{\op}$-module structure on $R_i$ is obtained via the homomorphism $R\to R_0$).
The condition (R3) is satisfied if $R$ has an $a$-invariant $a$,
but (R3) is much more general (see Remark~\ref{a-invariant} for details).

Under these assumptions, we apply Corollary~\ref{triangle equivalence2} to deduce the following result.

\begin{corollary}\label{orlov2}
Let $\ell\in\Z$. Under the assumptions {\rm(R0)}, {\rm(R2)} and {\rm(R3)}, we have
\[\cd^{\bo}(\mod^{\Z}R)=\cd^{\bo}(\mod_0^{\ge\ell}R)\perp\left(\cd^{\bo}(\mod^{\ge\ell}R)\cap\cd^{\bo}(\mod^{>a-\ell}R^{\op})^*\right)\perp\cd^{\bo}(\mod_0^{<\ell}R).\]
In particular, the composition
\[\cd^{\bo}(\mod^{\ge\ell}R)\cap\cd^{\bo}(\mod^{>a-\ell}R^{\op})^*\subset\cd^{\bo}(\mod^{\Z}R)\to\cd^{\bo}(\mod^{\Z}R)/\cd^{\bo}(\mod_0^{\Z}R)\]
is a triangle equivalence.
\end{corollary}

This is given implicitly in \cite[Lemma 2.4]{Orlov}, and a similar result is given in \cite[Theorem 5.16]{HIMO}.

To prove it, we need the following elementary observation.

\begin{lemma}[{c.f.\ \cite[Lemma 2.3]{Orlov}}]\label{SOD of S}
Let $\ell\in\Z$. Under the assumtion {\rm(R0)}, we have
$\cd^{\bo}(\mod_0^{\Z}R)=\cd^{\bo}(\mod_0^{\ge\ell}R)\perp\cd^{\bo}(\mod_0^{<\ell}R)$
and $\cd^{\bo}(\mod^{\Z}R)=\cd^{\bo}(\mod^{\ge\ell}R)\perp\cd^{\bo}(\mod_0^{<\ell}R)$.
\end{lemma}

\begin{proof}
Clearly $\Hom_{\cd^{\bo}(\mod^{\Z}R)}(\cd^{\bo}(\mod^{\ge\ell}R),\cd^{\bo}(\mod^{<\ell}R))=0$ holds.
For any $X\in\mod^{\Z}R$, let $X^{\ge\ell}=\bigoplus_{i\ge\ell}X_i\subset X$, and $X^{<\ell}:=X/X^{\ge\ell}$.
These give exact functors $(-)^{\ge\ell}:\mod^{\Z}R\to\mod^{\ge\ell}R$,
$(-)^{<\ell}:\mod^{\Z}R\to\mod^{<\ell}R=\mod_0^{<\ell}R$ and a sequence
$0\to(-)^{\ge\ell}\to {\rm id}\to(-)^{<\ell}\to0$
of natural transformations. Therefore we have induced triangle functors
$(-)^{\ge\ell}:\cd^{\bo}(\mod^{\Z}R)\to\cd^{\bo}(\mod^{\ge\ell}R)$,
$(-)^{<\ell}:\cd^{\bo}(\mod^{\Z}R)\to\cd^{\bo}(\mod_0^{<\ell}R)$ and a functorial triangle
$Y^{\ge\ell}\to Y\to Y^{<\ell}\to Y^{\ge\ell}[1]$ for any $Y\in\cd^{\bo}(\mod^{\Z}R)$.
Thus we have the second equality. The first one follows from the second one.
\end{proof}

\begin{proof}[Proof of Corollary \ref{orlov2}]
Without loss of generality, assume $\ell=0$.
Let $\ct:=\cd^{\bo}(\mod^{\Z}R)$, $\cs:=\cd^{\bo}(\mod_0^{\Z}R)$, $\cx:=\cd^{\bo}(\mod_0^{\ge0}R)$ and $\cy:=\cd^{\bo}(\mod_0^{<0}R)$.
By Lemma \ref{SOD of S}, we have stable t-structures
\begin{eqnarray*}
\cd^{\bo}(\mod_0^{\Z}R)=\cx\perp\cy\ \mbox{ and }\ \cd^{\bo}(\mod^{\Z}R)={}^\perp\cy\perp\cy
\end{eqnarray*}
with ${}^\perp\cy=\cd^{\bo}(\mod^{\ge0}R)$. Replacing $R$ by $R^{\op}$ and shifting the degree in the second stable t-structure, we have a stable t-structure
\begin{equation}\label{SOD}
\cd^{\bo}(\mod^{\Z}R^{\op})=\cd^{\bo}(\mod^{>a}R^{\op})\perp\cd^{\bo}(\mod_0^{\le a}R^{\op}).
\end{equation}
By (R3), the duality $(-)^*$ induces a duality $(-)^*:\cx\simeq\cd^{\bo}(\mod_0^{\le a}R^{\op})$.
Applying $(-)^*$ to \eqref{SOD}, we have a stable t-structure
\[\cd^{\bo}(\mod^{\Z}R)=\cd^{\bo}(\mod^{\le a}R^{\op})^*\perp\cd^{\bo}(\mod^{>a}R^{\op})^*=\cx\perp\cx^\perp.\]
Thus (T0), (T1$^{\prime\prime}$) and (T2) in Corollary~\ref{triangle equivalence2} are satisfied, and we have the assertion.
\end{proof}

Combining Corollaries \ref{orlov1} and \ref{orlov2}, we have the following Orlov's theorem immediately.

\begin{corollary}[{\cite[Theorem 2.5]{Orlov}}]\label{orlov3}
Under the assumptions {\rm(R0)}, {\rm(R1)}, {\rm(R2)} and {\rm(R3)}, we have the following.
\begin{itemize}
\item[(a)] If $a<0$, then there exists a fully faithful triangle functor
\[\cd^{\bo}(\mod^{\Z}R)/\ck^{\bo}(\proj^{\Z}R)\to\cd^{\bo}(\mod^{\Z}R)/\cd^{\bo}(\mod_0^{\Z}R).\]
\item[(b)] If $a=0$, then there exists a triangle equivalence
\[\cd^{\bo}(\mod^{\Z}R)/\ck^{\bo}(\proj^{\Z}R)\simeq\cd^{\bo}(\mod^{\Z}R)/\cd^{\bo}(\mod_0^{\Z}R).\]
\item[(c)] If $a>0$, then there exists a fully faithful triangle functor
\[\cd^{\bo}(\mod^{\Z}R)/\cd^{\bo}(\mod_0^{\Z}R)\to\cd^{\bo}(\mod^{\Z}R)/\ck^{\bo}(\proj^{\Z}R).\]
\end{itemize}
\end{corollary}

It is easy to show that the fully faithful functors in (a) and (c) are parts of stable t-structures.

%
%

\begin{remark}\label{a-invariant}
Assume (R0) and (R2).
\begin{itemize}
\item[(a)] The condition (R3) is clearly equivalent to that $\RHom_R(S,R)\in\cd^{\bo}(\mod^aR^{\op})$ holds for any simple $R$-module $S\in\mod^0R$ and $\RHom_{R^{\op}}(S',R)\in\cd^{\bo}(\mod^aR)$ holds for any simple $R^{\op}$-module $S'\in\mod^0R^{\op}$.
\item[(b)] The condition (R3) is satisfied if $R$ has \emph{$a$-invariant} (that is, the minus \emph{Gorenstein parameter}) $a$.
This means that there exists an integer $d$ such that, for any simple $R_0$-module $S$, there exists a simple $R_0^{\op}$-module $S'$ such that $S^*\simeq S'[-d](-a)$, and the same condition holds for simple $R_0^{\op}$-modules.
For example, this is satisfied if $R$ is ring-indecomposable and commutative.
\item[(c)] The condition (R3) is much weaker than the existence of $a$-invariant. 
For example, let $k$ be a separable field, and $A$ and $B$ be $\Z$-graded finite dimensional $k$-algebras. 
Assume that $A=A_0$ has finite global dimension and is not semisimple, and that $B$ is selfinjective and has an $a$-invariant.
Then $R:=A\otimes_kB$ is an Iwanaga--Gorenstein ring and satisfies the condition (R3), but does not has an $a$-invariant.
\end{itemize}
\end{remark}

\begin{proof}
(c) Since $k$ is separable, any simple $R$-module has the form $S\otimes_kT$ for a simple $A$-module $S$ and a simple $B$-module $T$.
Since $\RHom_R(S\otimes_kT,R)\simeq\RHom_A(S,A)\otimes_k\RHom_B(T,B)$ holds, the former statement follows.
It also follows that, if $R$ has an $a$-invariant, then so does $A$. On the other hand, it is easy to show that, if a $\Z$-graded finite dimensional $k$-algebra has an $a$-invariant, then it is selfinjective. 
Since $A$ has finite global dimension, it has to be semisimple, a contradiction.
\end{proof}

\subsection{The AGK category and Amiot--Guo--Keller's equivalence}\label{apply Amiot-Guo-Keller}

In cluster theory, an equivalence between the cluster category of an $n$-Calabi--Yau algebra $A$ and a certain full subcategory $\cz$ of the perfect derived category $\per A$ of $A$ plays a very important role (see a survey article \cite{Keller}). This was given by Amiot \cite{Amiot09} for $n=3$ and Guo \cite{Guolingyan11a} for general $n$ based on Keller's work \cite{Keller05, Keller11}.
The aim of this subsection is to deduce Amiot--Guo--Keller's equivalence from Theorem~\ref{equivalence}.

In fact, we will work in the following much wider setting. Let $k$ be a field and $D=\Hom_k(-,k)$.

\begin{definition}\label{serre quadruple}
 We say that $(\ct,\cs,\mathbb{S},\cm)$ is a \emph{relative Serre quadruple} if the following conditions are satisfied.
\begin{itemize}
\item[(RS0)] $\ct$ is a $k$-linear Hom-finite Krull--Schmidt triangulated category and $\cs$ is a thick subcategory of $\ct$.
\item[(RS1)] $\mathbb{S}:\cs\to\cs$ is a triangle equivalence such that there is a bifunctorial isomorphism for any $X\in\cs$ and $Y\in\ct$:
\[D\Hom_{\ct}(X,Y)\simeq \Hom_{\ct}(Y,\mathbb{S}X).\]
\item[(RS2)] $\cm$ is a silting subcategory of $\ct$ and $\ct=\cm[<\hspace{-3pt}0]^\perp\perp\cm[\ge\hspace{-3pt}0]^\perp$ is a t-structure of $\ct$ satisfying $\cm[\geq~\hspace{-3pt}0]^\perp\subset\cs$. Moreover, $\cm$ is a dualizing $k$-variety.
\end{itemize}
\end{definition}
Note that the last condition that $\cm$ is a dualizing $k$-variety is automatic if $\cm$ has an additive generator.
By \cite[Theorem 4.10]{IyamaYang},  (RS2) is equivalent to its dual:
\begin{itemize}
\item[(RS2$^{\op}$)] $\cm$ is a silting subcategory of $\ct$ and $\ct={}^\perp\cm[<\hspace{-3pt}0]\perp{}^\perp\cm[\ge\hspace{-3pt}0]$ is a t-structure of $\ct$ satisfying ${}^\perp\cm[<\hspace{-3pt}0]\subset\cs$. Moreover, $\cm$ is a dualizing $k$-variety.
\end{itemize}

Let us introduce the following new class of triangulated categories.

\begin{definition}\label{define amiot-guo-keller}
For a relative Serre quadruple $(\ct,\cs,\mathbb{S},\cm)$, we define the \emph{AGK category} as the Verdier quotient
\[\cc:=\ct/\cs.\]
We define the \emph{fundamental domain} as the full subcategory
\[\cz=\cx^\perp\cap{}^\perp\cy[1]\subset\ct,\ \mbox{ where }\ \cx=\cm[<\hspace{-3pt}0]^\perp\cap\cs\ \mbox{ and }\ \cy=\cm[\ge\hspace{-3pt}0]^\perp.\]
\end{definition}

The following theorem is the main result of this subsection.

\begin{theorem}\label{amiot-guo-keller}
Let $(\ct,\cs,\mathbb{S},\cm)$ be a relative Serre quadruple. Assume that one of the following conditions holds.
\begin{itemize}
\item[(i)] $\mathbb{S}:\cs\to\cs$ extends to a triangle equivalence $\mathbb{S}:\ct\to\ct$;
\item[(ii)] $\cm$ has an additive generator $M$. 
\end{itemize}
Then the composition
\[\cz\subset\ct\to\cc\]
is an equivalence. As a consequence, the AGK category $\cc$ is a Hom-finite Krull--Schmidt  triangulated category. Moreover, in case {\rm (i)}, $\cc$ has a Serre functor $\mathbb{S}\circ[-1]$.
\end{theorem}

\begin{proof}
By (RS2), we have a t-structure $\ct={}^\perp\cy\perp\cy$ with ${}^\perp\cy=\cm[<\hspace{-3pt}0]^\perp$. Therefore, for any $S\in\cs$, there exists a triangle $S'\to S\to Y\to S'[1]$ with $S'\in \cm[<\hspace{-3pt}0]^\perp$ and $Y\in\cy$. 
Since $S'$ belongs to $Y[-1]*S\in\cs*\cs=\cs$, it belongs to $\cx$. Thus we have a t-structure $\cs=\cx\perp\cy$. 

Now we show that $\ct=\cx\perp\cx^\perp$ is a t-structure in both cases (i) and (ii). Consequently, (T0), (T1) and (T2) in Theorem \ref{equivalence} are satisfied, and hence  the composition $\cz=\cx^\perp\cap{}^\perp\cy[1]\subset\ct\to\cc$ is an equivalence since $\cp=\cx[1]\cap\cy=0$.

First, we assume (i). Using the relative Serre duality (RS1) and the assumption ${}^\perp\cm[<\hspace{-3pt}0]\subset\cs$ in (RS2$^{\op}$), we have $\cx=\mathbb{S}({}^\perp\cm[<\hspace{-3pt}0])={}^\perp(\mathbb{S}\cm)[<\hspace{-3pt}0]$.
By (RS2$^{\op}$), we have a t-structure
$\ct={}^\perp(\mathbb{S}\cm)[<\hspace{-3pt}0]\perp{}^\perp(\mathbb{S}\cm)[\ge\hspace{-3pt}0]=\cx\perp\cx^\perp$ with $\cx^\perp={}^\perp(\mathbb{S}\cm)[\ge\hspace{-3pt}0]$.

Next we assume (ii). Using (RS2$^{\op}$), the dual argument to the first paragraph shows that we have t-structures $\ct=\cx'\perp\cx'{}^\perp$ and $\cs=\cx'\perp\cy'$ for $\cx'={}^\perp\cm[<\hspace{-3pt}0]$ and $\cy'=({}^\perp\cm[\ge\hspace{-3pt}0])\cap\cs$.

The t-structures $\cs=\cx\perp\cy$ and $\cs=\cx'\perp\cy'$ are bounded by \cite[Lemma 5.2]{IyamaYang}.
Hence the heart $\ch=\cx\cap\cy[1]$ is equivalent to the category of finite-dimensional modules over the finite dimensional $k$-algebra $\End_\ct(M)$ by \cite[Proposition 4.8]{IyamaYang}. 
Thus any object in the abelian category $\ch$ has finite length, and moreover $\ch$ contains only finitely many simple objects up to isomorphism. Let $L$ be the direct sum of a complete set of pairwise non-isomorphic simple objects of $\ch$.
Since $\cs=\cx'\perp\cy'$ is bounded, there exists $n\gg0$ such that $L[n]\in\cx'$. Then $\ch[n]\subset\cx'$ and hence $\cx[n]\subset\cx'$ holds.
By Lemma \ref{torsion pair criterion} below, we have a torsion pair $\ct=\cx\perp\cx^\perp$, which is a t-structure.

\begin{lemma}\label{torsion pair criterion}
Let $\ct$ be a triangulated category, $\cs$ a thick subcategory of $\ct$, and $\cs=\cx\perp\cy=\cx'\perp\cy'$ and $\ct=\cx'\perp\cx'{}^\perp$ torsion pairs.
If $\cx[n]\subset\cx'$ holds for some integer $n$, then we have a torsion pair $\ct=\cx\perp\cx^\perp$.
\end{lemma}

\begin{proof}
Replacing $\cx$ by $\cx[n]$, we can assume $n=0$.
It suffices to show $\cx*\cx^\perp\supset\ct$. This follows from
\begin{align*}
\cx*\cx^\perp
&=\cx*\cy*\cx^\perp \qquad (\cx^\perp=\cy*\cx^\perp)\\
&=\cs*\cx^\perp \qquad (\cx*\cy=\cs)\\
&\supset\cx'*\cx'^\perp\qquad (\cs\supset\cx',\ \cx^\perp\supset\cx'{}^\perp)\\
&=\ct.\qedhere
\end{align*}
\end{proof}

Let us continue the proof of Theorem~\ref{amiot-guo-keller}. It remains to prove the existence of Serre duality in case (i). Let $X$ and $Y$ be objects of $\ct$.
Since $\cm$ is a silting subcategory of $\ct$, we have a bounded co-t-structure $\ct={}^\perp\cm[\ge\hspace{-3pt}0]\perp\cm[<\hspace{-3pt}0]^\perp$ (see for example \cite[Proposition 2.8]{IyamaYang}).
It follows that there exists an integer $i$ such that $Y$ belongs to ${}^\perp\cm[\ge\hspace{-3pt}-i+1]$.
Now by (RS2) there is a triangle
\[
\xymatrix{
X'\ar[r] & X\ar[r] & X''\ar[r] &X'[1]
}
\]
with $X'\in \cm[<\hspace{-3pt}-i+1]^\perp$ and $X''\in \cm[\ge\hspace{-3pt}-i+1]^\perp$. Since $\Hom_\ct(Y,X')=0$, it follows that the induced homomorphism $\Hom_{\ct}(Y,X)\rightarrow\Hom_{\ct}(Y,X'')$ is injective.
So the morphism $X\rightarrow X''$ is a local $\cs$-envelope of $X$ relative to $Y$ in the sense of \cite[Definition 1.2]{Amiot09}.
Therefore by \cite[Lemma 1.1, Theorem 1.3 and Proposition 1.4]{Amiot09} we obtain that $\mathbb{S}\circ[-1]$ is a Serre functor of $\cc$.
\end{proof}

Let $(\ct,\cs,\mathbb{S},\cm)$ be a relative Serre quadruple and $n\geq 2$ an integer.
If $\mathbb{S}\simeq [n]$, then we call the triple $(\ct,\cs,\cm)$ an \emph{$n$-Calabi--Yau triple} and call $\cc=\ct/\cs$ the \emph{cluster category}  \cite[Section 5]{IyamaYang}.
A typical case is given by a bimodule $n$-Calabi--Yau non-positive dg algebra $A$ with $H^0(A)$ being finite-dimensional. In this case, $(\ct,\cs,\cm):=(\per A,\cd_{\rm fd}(A),\add A)$ is an $n$-Calabi--Yau triple (\cite[Section 2]{Amiot09}, \cite[Section 2]{Guolingyan11a}).

Let $(\ct,\cs,\cm)$ be an $n$-Calabi--Yau triple. Then
\[\cz=\cm[\le\hspace{-3pt}0]^\perp\cap{}^\perp \cm[\ge\hspace{-3pt}n]=\cm[1]*\cm[2]*\cdots*\cm[n-1].\]
As a special case of Theorem~\ref{amiot-guo-keller}, we obtain the following result which was obtained by Amiot \cite[Proposition 2.9]{Amiot09} and Guo 
\cite[Proposition 2.15]{Guolingyan11a} for bimodule Calabi--Yau dg algebras and by the authors in the general setting in \cite{IyamaYang}. It plays a crucial role in cluster theory.

\begin{corollary}[{\cite[Theorem 5.8(b)]{IyamaYang}}]\label{amiot-guo-keller-CY}
For an $n$-Calabi--Yau triple $(\ct,\cs,\cm)$, the composition
\[\cz\subset\ct\to\cc\]
is an equivalence. Moreover $\cc$ is an $(n-1)$-Calabi-Yau triangulated category.
\end{corollary}

\section{Proof of Main results}\label{s:proofs}

In this section we prove our main results Theorems~\ref{equivalence} and \ref{triangle equivalence}.

\subsection{Proof of Theorem \ref{equivalence}}
In this subsection, we prove Theorem \ref{equivalence}.
First, notice that
\[\cs\cap\cz=(\cs\cap\cx^\perp)\cap(\cs\cap{}^\perp\cy[1])=\cy\cap\cx[1]=\cp\]
hold by (T1). Thus the functor $\cz\to\cu$ induces a functor
\begin{equation}\label{Z/P to U}
\cz/[\cp]\to\cu.
\end{equation}
Next we show that this is dense.

\begin{lemma}\label{dense}
For any $X\in\ct$, there exists $Y\in\cz$ satisfying $X\simeq Y$ in $\cu$.
As a consequence, the functor \eqref{Z/P to U} is dense.
\end{lemma}

\begin{proof} Let $T\in\cu$. 
Since $\ct={}^\perp\cy[1]\perp\cy[1]$ holds by (T2), we have a triangle
\[
\xymatrix{
T'\ar[r] & T\ar[r] & Y[1]\ar[r] & T'[1] & (T'\in{}^\perp\cy[1],\ Y\in\cy).
}
\]
Then we have $T\simeq T'$ in $\cu$. Since $\ct=\cx\perp\cx^\perp$ holds again by (T2), we have a triangle
\[
\xymatrix{
X\ar[r] & T'\ar[r] & T''\ar[r] & X[1] & (X\in\cx,\ T''\in\cx^\perp).
}
\]
Then we have $T\simeq T'\simeq T''$ in $\cu$. Since both $T'$ and $X[1]$ belongs to ${}^\perp\cy[1]$ by (T1), so does $T''$.
Thus $T''$ belongs to $\cx^\perp\cap{}^\perp\cy[1]=\cz$, and we have an isomorphism $T\simeq T''$ in $\cu$.
\end{proof}

Next we prepare the following.

\begin{lemma}\label{cp}
We have $\cx[1]\subset\cx\perp\cp$ and $\cy\subset\cp\perp\cy[1]$.
\end{lemma}

\begin{proof}
We only prove the first assertion.
For $X\in\cx$, we take a triangle $X'\to X[1]\to Y\to X'[1]$ with $X'\in\cx$ and $Y\in\cy$ by (T1).
Since $\cx[1]$ is extension closed, we have $Y\in\cx[1]\cap\cy=\cp$.
Thus $X[1]\in\cx*\cp=\cx\perp\cp$.
\end{proof}

Finally we show that our functor is fully faithful.

\begin{lemma}\label{fully faithful}
The functor \eqref{Z/P to U} is fully faithful.
\end{lemma}

\begin{proof}
For $M,N\in\cz$, we consider the natural map $\Hom_{\cz/[\cp]}(M,N)\to\Hom_{\cu}(M,N)$.

We first show the surjectivity. 

Any morphism in $\Hom_{\cu}(M,N)$ has a representative of the form $M\xrightarrow{f}T\xleftarrow{s}N$, where $f\in\Hom_{\ct}(M,T)$ and $s\in\Hom_{\ct}(N,T)$, such that the cone of $s$ is in $\SS$.
Take a triangle
\[
\xymatrix{
N\ar[r]^{s} & T\ar[r] & S\ar[r]^(0.4){a} & N[1] & (S\in\SS).
}
\]
By (T1), there exists a triangle
\[
\xymatrix{
X[1]\ar[r]^{b} &S\ar[r] &Y[1]\ar[r] & X[2] & (X\in\cx,\ Y\in\cy).
}
\]
Since $ab=0$ by $X\in\cx$ and $N\in\cz\subset\cx^\perp$, we have the following commutative diagram of triangles by the octahedral axiom.
\[\xymatrix{
&X[1]\ar@{=}[r]\ar[d]&X[1]\ar[d]^b\\
N\ar[r]^s\ar@{=}[d]&T\ar[r]\ar[d]^c&S\ar[r]^a\ar[d]&N[1]\ar@{=}[d]\\
N\ar[r]^{cs}&T'\ar[r]^d\ar[d]&Y[1]\ar[r]\ar[d]&N[1]\\
&X[2]\ar@{=}[r]&X[2]
}\]
Then we have $dcf=0$ by $M\in\cz\subset{}^\perp\cy[1]$ and $Y\in\cy$.
Thus there exists $e\in\Hom_{\ct}(M,N)$ such that $cf=cse$.
Now $c(f-se)=0$ implies that $f-se$ factors through $X[1]\in\cs$.
Thus $f=se$ and $s^{-1}f=e$ hold in $\cu$, and we have the assertion.

\smallskip
Next we show the injectivity.

Assume that a morphism $f\in\Hom_{\ct}(M,N)$ is zero in $\cu$. 
Then it factors through $\cs$ (by, for example, \cite[Lemma 2.1.26]{Neeman01b}), that is, there exist $S\in\cs$, $g\in\Hom_{\ct}(M,S)$ and
$a\in\Hom_{\ct}(S,N)$ such that $f=ag$.
By (T1), there exists a triangle
\[
\xymatrix{
X\ar[r]^(0.55){b} & S\ar[r]^(0.4){c} & Y\ar[r] & X[1] & (X\in\cx,\ Y\in\cy).
}
\]
Since $ab=0$ by $X\in\cx$ and $N\in\cz\subset\cx^\perp$, there exists $d\in\Hom_{\ct}(Y,N)$ such that $a=dc$.
\[\xymatrix{
X\ar[r]^(0.55){b} & S\ar[r]^(0.4){c}\ar[dr]_a & Y\ar[d]^d\\
M\ar[rr]_f\ar[ru]_g&&N
}
\]
By Lemma \ref{cp}, there exists a triangle
\[
\xymatrix{
P\ar[r] & Y\ar[r]^(0.47){e} & Y'[1]\ar[r] & P[1] & (P\in\cp,\ Y'\in\cy).
}
\]
Then we have $ecg=0$ by $M\in\cz\subset{}^\perp\cy[1]$ and $Y'\in\cy$.
Thus $cg$ factors through $P$, and $f=dcg=0$ in $\cz/[\cp]$.
\end{proof}

Theorem \ref{equivalence} now follows, being the union of Lemmas \ref{dense} and \ref{fully faithful}.
\qed

\subsection{Proof of Theorem \ref{triangle equivalence}}\label{proof of Theorem2}

In this subsection, we prove Theorem \ref{triangle equivalence}. We start with the following observations.

\begin{lemma}\label{mutation pair}
Under the assumptions {\rm(T0)}, {\rm(T1$^\prime$)} and {\rm(T2)}, the following assertions hold.
\begin{itemize}
\item[(a)] $\cp\subset\cz$ and $\Hom_{\ct}(\cp,\cz[1])=0=\Hom_{\ct}(\cz,\cp[1])$ hold. 
\item[(b)] For any $Z\in\cz$, there exists a triangle $Z'\xrightarrow{a}P\xrightarrow{b}Z\to Z'[1]$ (respectively, $Z\xrightarrow{a}P\xrightarrow{b}Z'\to Z[1]$) with $Z'\in\cz$ and $P\in\cp$ such that $a$ is a left $\cp$-approximation and $b$ is a right $\cp$-approximation.
\end{itemize}
\end{lemma}


\begin{proof}
(a) These are clear.

(b) By (T2), there exists a triangle
\[T\xrightarrow{a} X[1]\xrightarrow{b} Z\to T[1]\]
with $X\in\cx$ and $T\in\cx^\perp$. Applying $\Hom_{\ct}(-,\cy[1])$, we have an exact sequence
\[0=\Hom_{\ct}(X[1],\cy[1])\to\Hom_{\ct}(T,\cy[1])\to\Hom_{\ct}(Z,\cy[2])\]
where the right term is zero since $Z\in\cz$ and $\cy[2]\subset\cy[1]$ holds by (T1$^\prime$). Thus $T\in\cx^\perp\cap{}^\perp\cy[1]=\cz$.

Since $T,Z\in\cz$ and $\cz$ is extension-closed, we have $X[1]\in\cx[1]\cap\cz\subset\cs\cap\cz=\cp$.
It is clear that $a$ is a left $\cp$-approximation and that $b$ is a right $\cp$-approximation since $\Hom_{\ct}(Z,\cp[1])=0=\Hom_{\ct}(\cp,Z[1])$ holds by (a).
\end{proof}

Now we are ready to prove Theorem \ref{triangle equivalence}.

(a) By Lemma \ref{mutation pair}, the pair $(\cz,\cz)$ forms a $\cp$-mutation pair in the sense of \cite{IyamaYoshino08}.
In particular, by \cite[Theorem 4.2]{IyamaYoshino08}, the category $\cz/[\cp]$ has the structure of a triangulated category
with respect to the shift functor and triangles given in the statement.
Moreover, it is easy to check that, with respect to this triangulated structure of $\cz/[\cp]$,
the equivalence $\cz/[\cp]\simeq\cu$ in Theorem \ref{equivalence} is a triangle functor.
Thus the assertion follows.

(b) It suffices to prove ${}^\perp\cy[1]=\cx\perp\cz$.

By (T1$^\prime$), we have ${}^\perp\cy[1]\supset{}^\perp\cy\supset\cx$. Thus ${}^\perp\cy[1]\supset\cx\perp\cz$ holds. It remains to show ${}^\perp\cy[1]\subset\cx\perp\cz$.
For any $T\in{}^\perp\cy[1]$, there exists a triangle
\[X\to T\to T'\to X[1]\]
with $X\in\cx$ and $T'\in\cx^\perp$ by (T2). Since $X[1]\in\cx[1]\subset{}^\perp\cy[1]$, we have $T'\in\cx^\perp\cap{}^\perp\cy[1]=\cz$.
Thus $T\in\cx\perp\cz$, and we have the assertion.
\qed

\def\cprime{$'$}
\providecommand{\bysame}{\leavevmode\hbox to3em{\hrulefill}\thinspace}
\providecommand{\MR}{\relax\ifhmode\unskip\space\fi MR }
\providecommand{\MRhref}[2]{%
  \href{http://www.ams.org/mathscinet-getitem?mr=#1}{#2}
}
\providecommand{\href}[2]{#2}

\end{document}